\documentclass[oneside,english]{amsart}
\usepackage[T1]{fontenc}
\usepackage[latin9]{inputenc}
\usepackage{color}
\usepackage{babel}
\usepackage{textcomp}
\usepackage{mathrsfs}
\usepackage{amstext}
\usepackage{amsthm}
\usepackage{amssymb}
\usepackage{stmaryrd}
\usepackage[all]{xy}
\usepackage[unicode=true,pdfusetitle,
 bookmarks=true,bookmarksnumbered=false,bookmarksopen=false,
 breaklinks=true,pdfborder={0 0 0},pdfborderstyle={},backref=false,colorlinks=true]
 {hyperref}

\makeatletter
\theoremstyle{plain}
\newtheorem{thm}{\protect\theoremname}[section]
\theoremstyle{plain}
\newtheorem{conjecture}{\protect\conjecturename}[section]
\theoremstyle{plain}
\newtheorem*{conjecture*}{\protect\conjecturename}
\theoremstyle{remark}
\newtheorem*{notation*}{\protect\notationname}
\theoremstyle{plain}
\newtheorem{prop}{\protect\propositionname}[section]
\theoremstyle{remark}
\newtheorem*{rem*}{\protect\remarkname}
\theoremstyle{plain}
\newtheorem{cor}{\protect\corollaryname}[section]
\theoremstyle{plain}
\newtheorem*{fact*}{\protect\factname}
\theoremstyle{plain}
\newtheorem*{thm*}{\protect\theoremname}
\theoremstyle{plain}
\newtheorem*{lem*}{\protect\lemmaname}
\theoremstyle{definition}
\newtheorem*{example*}{\protect\examplename}
\theoremstyle{plain}
\newtheorem{lem}{\protect\lemmaname}[section]

\usepackage[pdftex]{pict2e} 
\usepackage[OT2,OT1]{fontenc}
\newcommand\cyr{%
\renewcommand\rmdefault{wncyr}%
\renewcommand\sfdefault{wncyss}%
\renewcommand\encodingdefault{OT2}%
\normalfont
\selectfont}
\DeclareTextFontCommand{\textcyr}{\cyr}

\DeclareMathOperator{\lcm}{lcm}
\DeclareMathOperator{\rank}{rank}

\DeclareMathOperator{\tors}{tors}
\def\ltriangle{\mbox{
\begin{picture}(3,4) 
\put(0,0){\line(1,0){6}} 
\put(6,0){\line(0,1){8}} 
\put(0,0){\line(3,4){6}} 
\end{picture}
}}
\newcommand{\sha}{\textrm{{\cyr SH}}}

\makeatother

\providecommand{\conjecturename}{Conjecture}
\providecommand{\corollaryname}{Corollary}
\providecommand{\examplename}{Example}
\providecommand{\factname}{Fact}
\providecommand{\lemmaname}{Lemma}
\providecommand{\notationname}{Notation}
\providecommand{\propositionname}{Proposition}
\providecommand{\remarkname}{Remark}
\providecommand{\theoremname}{Theorem}

\begin{document}
\title{Reflecting Numbers of Various Types, I}
\author{Ya-Qing Hu}
\address{Morningside Center of Mathematics\\
Chinese Academy of Sciences}
\email{\href{mailto:yaqinghu@amss.ac.cn}{yaqinghu@amss.ac.cn}}
\keywords{Reflecting Numbers, Diophantine equations, Congruent Numbers}
\subjclass[2010]{11D72, 11G05}
\thanks{This work is partially supported by the Postdoctoral International
Exchange Program of the China Postdoctoral Council with grant number
YJ20210319.}
\begin{abstract}
The purpose of this paper is to introduce the concept of reflecting
numbers to the realm of number theory and to classify reflecting numbers
of certain types.  For us, reflecting numbers are coming from congruent
numbers, above congruent numbers, and away from congruent numbers.

Explicitly speaking, a reflecting number of type $(k,m)$ is the average
of two distinct rational $k$th powers,  between which the distance
is twice another nonzero rational $m$th power. In particular, reflecting
numbers of type $(2,2)$ are all congruent numbers and thus will be
called reflecting congruent numbers in this paper. We can show that
all prime numbers $p\equiv5\mod8$ are reflecting congruent and in
general for any integer $k\ge0$ there are infinitely many square-free
reflecting congruent numbers in the residue class of $5$ modulo $8$
with exactly $k+1$ prime divisors. Moreover, we conjecture that all
prime congruent numbers $p\equiv1\mod8$ are reflecting congruent.
In addition, we show that there are no reflecting numbers of type
$(k,m)$ if $\gcd(k,m)\ge3$. 
\end{abstract}

\date{\today}

\maketitle
\tableofcontents{}

\section{Introduction}

Let $\mathcal{P}$ be a property about rational numbers, such as being
rational squares or cubes, and denote also by $\mathcal{P}$ the subset
of all rational numbers with this property.

A nonzero integer $n$ is called a \emph{reflecting number of type
$(\mathcal{P}_{1},\mathcal{P}_{2})$}, if there exist $u,v\in\mathcal{P}_{1}$
and $t\in\mathcal{P}_{2}\setminus\{0\}$ such that $n-t=u$ and $n+t=v$,
where $\mathcal{P}_{1}$ and $\mathcal{P}_{2}$ are two properties
about rational numbers. 

On a number line, the integer $n$ behaves like a mirror and reflects
two rational numbers $u,v$ with the same property $\mathcal{P}_{1}$
to each other, while the distance from $u$ or $v$ to $n$ is a nonzero
rational number $t$ with property $\mathcal{P}_{2}$, hence the name.
Since the meaning behind $t$ is distance, we require it to be positive.
Also, we exclude $n=0$ from the definition for it is in general less
interesting or can be easily dealt with.

In this work, we consider the property $\mathcal{P}(i)=\{x^{i}:x\in\mathbb{Q}\}$
consisting of rational $i$th powers. For any ordered pair of positive
integers $(k,m)$, reflecting numbers of type $(\mathcal{P}(k),\mathcal{P}(m))$
are also called $(k,m)$-reflecting numbers for short. So a \emph{$(k,m)$-reflecting
number} is the average of two distinct rational $k$th powers, between
which the distance is twice another nonzero rational $m$th power. 

A basic problem is to classify reflecting numbers $n$ of various
types $(k,m)$ and for each such number $n$ find all positive rational
numbers $t$ as in the definition. For example, if $k=1$ then $n\pm t^{m}$
are always rational numbers for any $t\in\mathbb{Q}^{*}$ and thus
all nonzero integers are $(1,m)$-reflecting for all natural number
$m$. 

Since $2n=v^{k}+u^{k}$ and $2t^{m}=v^{k}-u^{k}$, we observe that: 

1) If $k$ is odd, then $n$ is $(k,m)$-reflecting if and only if
$-n$ is. 

2) If $k$ is even, then negative integers cannot be $(k,m)$-reflecting. 

3) If $n$ is $(k,m)$-reflecting, so is $nd^{\lcm(k,m)}$ for any
positive integer $d$. So it suffices to study the \emph{primitive
$(k,m)$-reflecting numbers}, i.e., the ones that are positive and
free of $\lcm(k,m)$th power divisors. Denote by $\mathscr{R}(k,m)$
the set of all $(k,m)$-reflecting numbers and by $\mathscr{R}'(k,m)$
the subset of all primitive ones. 

4) If $n$ is $(k',m')$-reflecting, then it is $(k,m)$-reflecting
for all $k\mid k'$ and $m\mid m'$. So we have a filtration of reflecting
numbers of various types $(k,m)$, i.e,
\[
\mathscr{R}(k,m)\supset\mathscr{R}(k',m'),\ \forall k\mid k',\ \forall m\mid m'.
\]
In particular, if $\mathscr{R}(k,m)=\emptyset$ then $\mathscr{R}(k',m')=\emptyset$.

5) The set of $(k,m)$-reflecting numbers is nonempty if and only
if the ternary Diophantine equation $v^{k}-u^{k}=2t^{m}$ has a rational
solution such that $v^{k}\ne\pm u^{k}$.

6) Reflecting numbers of type $(k,1)$ are those nonzero integers,
whose double can be written as sums of two distinct rational $k$th
powers. 

The motivation comes from a sudden insight into the definition of
congruent numbers: a positive integer $n$ is call a \emph{congruent
number}, if it is the area of a right triangle with rational sides,
or equivalently, if it is the common difference of an arithmetic progression
of three rational squares, i.e., if there exists a positive rational
number $t$ such that $t^{2}\pm n$ are both rational squares. But
what happens if $t^{2}$ and $n$ are swapped in the latter definition?
Well, if there exists a positive rational number $t$ such that $n\pm t^{2}$
are rational squares, then $n$ is by our definition a\emph{ $(2,2)$-}reflecting
number, which turns out to be a congruent number. 

To emphasize that $(2,2)$-reflecting numbers are in fact congruent,
such numbers are also called \emph{reflecting congruent numbers}.
For example, $5$ is $(2,2)$-reflecting since $5-2^{2}=1^{2}$ and
$5+2^{2}=3^{2}$, and congruent since $(41/12)^{2}-5=(31/12)^{2}$
and $(41/12)^{2}+5=(49/12)^{2}$. But there exist congruent numbers,
such as $6$ and $7$, that are not reflecting congruent. Hence, we
have a dichotomy of congruent numbers, according to whether they are
reflecting congruent or not.

In this work, we focus on the primitive reflecting congruent numbers.
It is easy to see that such numbers can only have prime divisors congruent
to $1$ modulo $4$. Among all such numbers, the prime ones are certainly
interesting. Heegner \cite{Heegner1952} asserts without proof that
prime numbers in the residue class $5$ modulo $8$ are congruent
numbers. This result is repeated by Stephens \cite{Stephens1975}
and finally proved by Monsky \cite{Monsky1990}. But we can say more
about these numbers. 
\begin{thm}
\label{thm:p5mod8} Prime numbers $p\equiv5\mod8$ are reflecting
congruent.
\end{thm}
Most prime numbers congruent to $1$ modulo $8$ are not congruent.
But if they are congruent, then we have the following theorem and
conjecture.
\begin{thm}
\label{thm:p1mod8} Let $p$ be a prime congruent number $\equiv1\mod8$.
If $E_{p}/\mathbb{Q}$ has rank $2$ (or equivalently $\sha(E_{p}/\mathbb{Q})[2]$
is trivial), then $p$ is reflecting congruent.
\end{thm}
\begin{conjecture}
\label{conj:p1mod8} Prime congruent numbers $p\equiv1\mod8$ are
reflecting congruent.
\end{conjecture}
This conjecture follows easily from our criterion of reflecting congruent
numbers and any one of the parity conjecture for Mordell-Weil group
and the finiteness conjecture for Shafarevich-Tate group. 
\begin{conjecture*}
For any elliptic curve $E$ defined over a number field $K$, one
has $(-1)^{\rank(E/K)}=w(E/K)$, where $w(E/K)$ is the global root
number of $E$ over $K$.
\end{conjecture*}
For any square-free congruent number $n$ and the associated elliptic
curves $E_{n}:y^{2}=x^{3}-n^{2}x$ over $\mathbb{Q}$, the global
root number is given by the following formula
\begin{equation}
w(E_{n}/\mathbb{Q})=\begin{cases}
+1, & \text{if }n\equiv1,2,3\mod8,\\
-1, & \text{if }n\equiv5,6,7\mod8.
\end{cases}\label{eq:root number}
\end{equation}

\begin{conjecture*}
For any elliptic curve $E$ defined over a number field $K$, the
Shafarevich-Tate group $\sha(E/K)$ is finite.
\end{conjecture*}
A theorem of Tian \cite{Tian2012} says that for any given integer
$k\ge0$, there are infinitely many square-free congruent numbers
in each residue class of $5$, $6$, and $7$ modulo $8$ with exactly
$k+1$ odd prime divisors. With a bit of work, we can show 
\begin{thm}
\label{thm:TianThm1.1} For any given integer $k\ge0$, there are
infinitely many square-free reflecting congruent numbers in the residue
class of $5$ modulo $8$ with exactly $k+1$ prime divisors. 
\end{thm}
Moreover, this result can be strengthened by \cite[Thm. 5.2]{Tian2012}
as follows. 
\begin{thm}
\label{thm:TianThm5.2} Let $p_{0}\equiv5\mod8$ be a prime number.
Then there exists an infinite set $\Sigma$ of primes congruent to
$1$ modulo $8$ such that the product of $p_{0}$ with any finitely
many primes in $\Sigma$ is a reflecting congruent number. 
\end{thm}
In fact, we will show by our criterion of reflecting congruent numbers
that these congruent numbers congruent to $5$ modulo $8$ constructed
by Tian as in his \cite[Thm. 1.3]{Tian2012} are actually reflecting
congruent.
\begin{thm}
\label{thm:TianThm1.3} Let $n\equiv5\mod8$ be a square-free positive
integer with all prime divisors $\equiv1\mod4$, exactly one of which
is $\equiv5\mod8$, such that the field $\mathbb{Q}(\sqrt{-n})$ has
no ideal classes of exact order $4$. Then, $n$ is a reflecting congruent
number. 
\end{thm}
Before closing the paper, we discuss reflecting numbers of type $(k,m)$
with $\gcd(k,m)\ge3$. In virtue of a deep result on rational solutions
to the ternary Diophantine equation $x^{k}+y^{k}=2z^{k}$ for any
$k\ge3$, we can show that 
\begin{thm}
\label{thm:gcd(k,m)>=00003D3} There exist no reflecting numbers of
type $(k,m)$ if $\gcd(k,m)\ge3$. 
\end{thm}
Let $d$ be the greatest common divisor of $k$ and $m$. Since $\mathscr{R}(k,m)\subset\mathscr{R}(d,d)$,
it suffices to show that there exist no reflecting numbers of type
$(d,d)$ for any $d\ge3$. The cases in which $d=3,4$ can be easily
proved, due to classical results of Euler. The cases in which $d\ge5$
follow easily from the Lander, Parkin, and Selfridge conjecture \cite{LanderParkinSelfridge1967}
on equal sums of like powers or the D\'enes Conjecture \cite{Denes1952}
on arithmetic progressions of like powers:
\begin{conjecture*}
If the formula $\sum_{i=1}^{n}a_{i}^{d}=\sum_{j=1}^{m}b_{j}^{d}$
holds, where $a_{i}\neq b_{j}$ are positive integers for all $1\le i\le n$
and $1\le j\le m$, then $m+n\ge d$.
\end{conjecture*}
\begin{conjecture*}
For any $d\ge3$, if the ternary Diophantine equation $x^{d}+y^{d}=2z^{d}$
has a rational solution $(x,y,z)$, then $xyz=0$ or $|x|=|y|=|z|$. 
\end{conjecture*}
Fortunately, the latter D\'enes conjecture becomes a theorem by the
work of Ribet \cite{Ribet1997} and Darmon and Merel \cite{DarmonMerel1997},
based on which we prove Theorem \ref{thm:gcd(k,m)>=00003D3}. 

This work is organized as follows. Section \ref{sec:(k,m)} gives
a general description of $(k,m)$-reflecting numbers and picks out
some special ones, whose existence depends on whether $k$ and $m$
are coprime or not. Moreover, reflecting numbers of type $(k,1)$,
in particular, $(2,1)$ and $(3,1)$, are also discussed in this section.
Section \ref{sec:(2,2)} is devoted to reflecting congruent numbers
and occupies the major part of this work. In Section \ref{sec:gcd(k,m)>=00003D3},
we discuss reflecting numbers of type $(k,m)$ with $\gcd(k,m)\ge3$
and disprove their existence. 
\begin{notation*}
Throughout this paper, for any prime number $p$, we denote by $v_{p}$
the normalized $p$-adic valuation such that $v_{p}(\mathbb{Q}_{p}^{*})=\mathbb{Z}$
and $v_{p}(0)=+\infty$. 
\end{notation*}

\section{\label{sec:(k,m)}$(k,m)$-Reflecting Numbers}

In this section, we give a general discription of $(k,m)$-reflecting
numbers and then focus on $(k,1)$-reflecting numbers for smaller
$k=2,3$. 

\subsection{Homogeneous Equations}

By definition, a nonzero integer $n$ is $(k,m)$-reflecting if and
only if the following system of homogeneous equations 
\[
\begin{cases}
nS^{m}-T^{m}=S^{m-k}U^{k},\\
nS^{m}+T^{m}=S^{m-k}V^{k},
\end{cases}\text{if }k\le m,\text{ or }\begin{cases}
nS^{k}-S^{k-m}T^{m}=U^{k},\\
nS^{k}+S^{k-m}T^{m}=V^{k},
\end{cases}\text{if }k\ge m,
\]
has a nontrivial solution $(S,T,U,V)$ in integers such that $S,T>0$.
Such a solution is called \emph{primitive} if the greatest common
divisor of $S,T,U,V$ is $1$.\footnote{However, if $k\ne m$, then the greatest common divisor of $S$ and
$T$ might be greater than $1$.}  

In either case, we have $n=\frac{V^{k}+U^{k}}{2S^{k}}$ and $\frac{T^{m}}{S^{m}}=\frac{V^{k}-U^{k}}{2S^{k}}$.
Since $n$ and $T$ are nonzero, $V$ and $U$ must have the same
parity and satisfy $V^{k}\ne\pm U^{k}$.

\subsection{Inhomogeneous Equations}

If $k\ne m$, then it is more convenient to work with inhomogeneous
equations. Again we assume $t>0$ and write $t=T/S$ with $\gcd(S,T)=1$
and $S,T>0$. Then, we have 
\[
nS^{m}-T^{m}=S^{m}u^{k},\quad nS^{m}+T^{m}=S^{m}v^{k}.
\]
It follows that denominators of $u^{k}$ and $v^{k}$ are divisors
of $S^{m}$. If $p$ is a prime divisor of $S\ne1$, then $p$ does
not divide $T$. If $uv\ne0$, then we compare the $p$-adic valuations
of both sides of the above equations 
\[
mv_{p}(S)+kv_{p}(u)=v_{p}(nS^{m}-T^{m})=0=v_{p}(nS^{m}+T^{m})=mv_{p}(S)+kv_{p}(v).
\]
This implies that $S^{m}$ is a $k$th power and thus a $\lcm(k,m)$th
power, which is exactly the denominator of $u^{k}$ and $v^{k}$.
Then, for some nonzero integer $S_{0}$, we can write
\[
S=S_{0}^{k'},\ U^{k}=S^{m}u^{k}=(S_{0}^{m'}u)^{k},\ V^{k}=S^{m}v^{k}=(S_{0}^{m'}v)^{k}
\]
where $U,V\in\mathbb{Z}$ and $k'=k/\gcd(k,m),m'=m/\gcd(k,m)$. Therefore,
we have the following inhomogeneous equations, 
\begin{equation}
nS_{0}^{\lcm(k,m)}-T^{m}=U^{k},\quad nS_{0}^{\lcm(k,m)}+T^{m}=V^{k},\label{eq:inhomo}
\end{equation}
which become homogeneous when $k=m$. If exactly one of $u$ and $v$
is $0$ or $S=1$, then $t=T$, $S=S_{0}=1$, $u=U$, and $v=V$ are
all integers, so we will obtain the same equations as above. So we
can always write 
\[
n=\frac{V^{k}+U^{k}}{2S_{0}^{\lcm(k,m)}},\ T=\sqrt[m]{\frac{V^{k}-U^{k}}{2}},\text{ where }V^{k}\ne\pm U^{k}\text{ and }V\equiv U\mod2.
\]

\begin{prop}
\label{prop:R(k,m)} The set of $(k,m)$-reflecting numbers is given
by 
\[
\mathscr{R}(k,m)=\left\{ \frac{V^{k}+U^{k}}{2S_{0}^{\lcm(k,m)}}\in\mathbb{Z}^{*}:(S_{0},U,V)\in\mathbb{N}\times\mathbb{Z}^{2},\sqrt[m]{\frac{V^{k}-U^{k}}{2}}\in\mathbb{N}\right\} 
\]
and it is nonempty if and only if the following ternary Diophantine
equation 
\begin{equation}
2T^{m}+U^{k}=V^{k}\label{eq:2Tm+Uk=00003DVk}
\end{equation}
has an integer solution $(T,U,V)$ such that $U^{k}\ne\pm V^{k}$.
\end{prop}
\begin{proof}
It remains to show the ``if'' part. If (\ref{eq:2Tm+Uk=00003DVk})
has an integer solution $(T,U,V)$ such that $U^{k}\ne\pm V^{k}$,
then $U\equiv V\mod2$ and $n=(V^{k}+U^{k})/2$ is $(k,m)$-reflecting. 
\end{proof}

\subsection{Special $(k,m)$-Reflecting Numbers}

If $\gcd(k,m)=1$ and $i\in[0,k)$ is the least integer such that
$k\mid(im+1)$, then $n=2^{im}T_{0}^{km}$ is $(k,m)$-reflecting
for any $T_{0}\in\mathbb{N}$. Indeed, $T=2^{i}T_{0}^{k}$ and $V=2^{(im+1)/k}T_{0}^{m}$
are solutions to $n-T^{m}=0$ and $n+T^{m}=V^{k}$. This shows that
$\mathscr{R}(k,m)$ is not empty whenever $\gcd(k,m)=1$. We will
refer to such numbers (and their opposites if $k$ is odd) as the
\emph{special $(k,m)$-reflecting numbers}. 

It turns out that special $(k,m)$-reflecting numbers exist only when
$\gcd(k,m)=1$.
\begin{prop}
\label{prop:special} If $n=T^{m}$ is a positive $(k,m)$-reflecting
number such that $2T^{m}=V^{k}$ for some positive integers $T$ and
$V$, then $\gcd(k,m)=1$ and $n=2^{im}T_{0}^{km}$, where $T_{0}$
is a positive integer and $i$ is the least nonnegative integer such
that $k\mid(im+1)$.
\end{prop}
\begin{proof}
Indeed, the $2$-adic valuation $1+mv_{2}(T)=kv_{2}(V)$ of $2T^{m}=V^{k}$
shows that $\gcd(k,m)=1$. If $p$ is an odd prime divisor of $T$,
then we have $mv_{p}(T)=kv_{p}(V)$. Since $m\mid v_{p}(V)$ and $k\mid v_{p}(T)$,
we can write $T=2^{i}T_{0}^{k}$ where $T_{0}$ is a positive integer
and $i=v_{2}(T)\mod k\in[0,k)$ is the least integer such that $k\mid(im+1)$.
\end{proof}
\begin{rem*}
Special $(k,m)$-reflecting numbers give an abundant but less interesting
supply of solutions to the problem. One could avoid them by setting
$uv\ne0$ in the definition.
\end{rem*}

\subsection{$(k,1)$-Reflecting Numbers}

If $m=1$, then (\ref{eq:2Tm+Uk=00003DVk}) has integer solutions
if and only if $V$ and $U$ have the same parity. So we obtain that 
\begin{cor}
The set of $(k,1)$-reflecting numbers is given by 
\[
\mathscr{R}(k,1)=\left\{ \frac{V^{k}+U^{k}}{2S_{0}^{k}}\in\mathbb{Z}^{*}:(S_{0},U,V)\in\mathbb{N}\times\mathbb{Z}^{2},V^{k}\ne\pm U^{k},V\equiv U\mod2\right\} .
\]
\end{cor}
\begin{notation*}
For each $i\in\mathbb{N}$, the $i$th power free part of any nonzero
rational number $t$ is temporarily denoted by $\llbracket t\rrbracket_{i}$,
i.e., an integer such that $t/\llbracket t\rrbracket_{i}\in(\mathbb{Q}^{+})^{i}$
is a rational $i$th power. So $n=\llbracket n\rrbracket_{i}$ means
that $n$ is an integer having no $i$th power divisors.
\end{notation*}
The following fact is useful in our analysis of $(2,m)$-reflecting
numbers. 
\begin{fact*}
$-1$ is a quadratic residue modulo a positive square-free number
$n$ if and only if each odd prime divisor of $n$ is congruent to
$1$ modulo $4$. 
\end{fact*}
It is easy to classify primitive $(2,1)$-reflecting numbers, among
which $n=1$ is the first one and $1\cdot5^{2}-2^{3}\cdot3=1^{2}$
and $1\cdot5^{2}+2^{3}\cdot3=7^{2}$, where $S_{0}=5$ is the smallest,
and $n=2$ is the special one for which $u=0$ and $2-2=0^{2}$, $2+2=2^{2}$.
\begin{prop}
\label{prop:R'(2,1)} The set $\mathscr{R}'(2,1)$ consists of positive
square-free integers $n$, such that $n$ has no prime divisors congruent
to $3$ modulo $4$, or equivalently, $-1$ is a quadratic residue
modulo $n$. 
\end{prop}
\begin{proof}
If $n\in\mathscr{R}'(2,1)$, then $2nS_{0}^{2}=V^{2}+U^{2}$ is a
sum of two distinct squares with the same parity. Since the square-free
part of a sum of two squares cannot have a prime divisor $\equiv3\mod4$,
$n$ has no prime divisors $\equiv3\mod4$. Conversely, we have seen
that $1\in\mathscr{R}'(2,1)$ and if $n>1$ is a square-free integer
having no prime divisors $\equiv3\mod4$, then $2n$ can be written
as a sum of two squares with same parity. These two squares must be
distinct; otherwise, $n$ will not be square-free. 
\end{proof}
As a $(2,1)$-reflecting number, $n=1$ is closely related to the
congruent numbers.
\begin{prop}
There is a one-to-one correspondence between the set of square-free
congruent numbers and the set $\{\llbracket t\rrbracket_{2}\mid t\in\mathbb{Q}^{+},1\pm t\in\mathbb{Q}^{2}\}$.
\end{prop}
\begin{proof}
If $m$ is a square-free congruent number, then there exists $t\in\mathbb{Q}^{+}$
such that $t^{2}\pm m$ or $1\pm mt^{-2}$ are rational squares. Hence,
$1$ is $(2,1)$-reflecting and $m=\llbracket mt^{-2}\rrbracket_{2}$.
 Conversely, if $t\in\mathbb{Q}^{+}$ is such that $1-t=u^{2}$ and
$1+t=v^{2}$ for $u,v\in\mathbb{Q}^{+}$. Clearing denominators of
$t,u,v$, we have $S^{2}-S^{2}t=U^{2}$ and $S^{2}+S^{2}t=V^{2}$,
i.e., $S^{2}t$ is the common difference of the arithmetic progression
$U^{2},S^{2},V^{2}$ of perfect squares, and thus $\llbracket t\rrbracket_{2}=\llbracket S^{2}t\rrbracket_{2}$
 is a congruent number. 
\end{proof}

\begin{rem*}
Searching all rational numbers $t$ such that $1\pm t$ are rational
squares is equivalent to searching all congruent numbers $n$ and
arithmetic progressions of three rational squares with common difference
$n$. So in this sense the reflecting number problem generalizes the
congruent number problem.
\end{rem*}
It is not easy to classify all primitive $(3,1)$-reflecting numbers,
among which $3$ is the first one and $3\cdot21^{3}-22870=17^{3}$,
$3\cdot21^{3}+22870=37^{3}$, where $S_{0}=21$ is the smallest, and
$4$ is the special one and $4-4=0$, $4+4=2^{3}$. 

Note that $(3,1)$-reflecting numbers are closely related to sums
of two distinct rational cubes. Indeed, if $n$ is $(3,1)$-reflecting,
then there exists $u,v\in\mathbb{Q}$ such that $2n=u^{3}+v^{3}$
where $v\ne\pm u$, i.e., $2n$ can be written as sums of two distinct
rational cubes. Conversely, if $N\ne0$ can be written as sums of
two distinct rational cubes, so does $8N$ and thus $n=4N$ is $(3,1)$-reflecting.
The problem of deciding whether an integer can be written as sums
of two rational cubes has a long history; cf. \cite[Ch. XXI, pp. 572--578]{Dickson2005history}.
In particular, we recall the following theorem of Euler, cf. \cite[Ch. XV, Thm. 247, pp. 456--458]{Euler1822elements},
which proves that $n=1$ is not $(3,1)$-reflecting. 
\begin{thm*}
Neither the sum nor the difference of two cubes can become equal to
the double of another cube; or, in other words, the formula, $x^{3}\pm y^{3}=2z^{3}$,
is always impossible, except in the evident case of $x=y$.\footnote{ Here, $x$, $y$, and $z$ are assumed to be non-negative. }
\end{thm*}
\begin{rem*}
Euler's proof is based on Fermat's infinite descent method  and similar
to his proof of Fermat's Last Theorem of degree $3$. 
\end{rem*}
Let $C^{N}:u^{3}+v^{3}=N$ and $C_{N}:y^{2}=x^{3}+N$, where $N\ne0$,
be elliptic curves defined over rational numbers. The Weierstrass
form of the elliptic curve $C^{N}$ is given by the elliptic curve
$C_{-432N^{2}}$. Indeed, the homogeneous form of $C^{N}$ is given
by $U^{3}+V^{3}=NW^{3}$, which has a rational point $[1,-1,0]$.
One checks that 
\[
x=12N\frac{1}{v+u},\quad y=36N\frac{v-u}{v+u}
\]
satisfy the Weierstrass equation $C_{-432N^{2}}:y^{2}=x^{3}-432N^{2}$,
whose homogeneous form is given by $Y^{2}Z=X^{3}-432N^{2}Z^{3}$,
which contains a rational point $[0,1,0]$. Conversely, $u$ and $v$
can be expressed in terms of $x$ and $y$ by 
\[
u=\frac{36N-y}{6x},\quad v=\frac{36N+y}{6x}.
\]
Hence, there exists a one-to-one correspondence between rational points
$(u,v)\in C^{N}(\mathbb{Q})$ and rational points $(x,y)\in C_{-432N^{2}}(\mathbb{Q})$
(except the case when $v=-u$ we manually make the correspondence
between $[1,-1,0]$ and $[0,1,0]$). 

Let $N=2n$. Notice that $C_{-1728n^{2}}:y^{2}=x^{3}-1728n^{2}$ is
isomorphic to $C_{-27n^{2}}:y'^{2}=x'^{2}-27n^{2}$ via $y'=2^{-3}y$
and $x'=2^{-2}x$. 
\begin{prop}
\label{prop:R'(3,1):order} The set $\mathscr{R}'(3,1)$ consists
of positive cube-free integers $n$ such that there exists a nontrivial
rational point $(x,y)$ of order other than $2$ on $C_{-27n^{2}}$.
\end{prop}
\begin{proof}
 If $n\in\mathscr{R}'(3,1)$, then for some $v\neq\pm u\in\mathbb{Q}$
we have $v^{3}+u^{3}=2n$. Hence, $(u,v)$ is a rational point on
$C^{2n}$. Since $v\ne-u$, $(u,v)$ corresponds to $(x,y)=(6n\frac{1}{v+u},9n\frac{v-u}{v+u})$.
Since $v\ne u$ if and only if $y\ne0$, $(x,y)$ is a nontrivial
rational point on $C_{-27n^{2}}$ of order other than $2$.
\end{proof}
\begin{rem*}
Among all primitive $(3,1)$-reflecting numbers, $4$ is the special
one. The elliptic curve $C_{-432}$ has rank $0$ and its torsion
subgroup is isomorphic to $\mathbb{Z}/3$, in which the nontrivial
torsion points are $(12,\pm36)$, corresponding to the only rational
solutions $(v,u)=(0,2),(2,0)$ to $v^{3}+u^{3}=8$. In fact, $4$
is the only one with this property, as shown by the following lemma
and corollary.
\end{rem*}
\begin{lem*}
Let $N\ne0$ be an integer having no $6$th power divisors. Then,
a complete description of the torsion subgroup $(C_{N})_{\tors}(\mathbb{Q})$
is given by 
\[
(C_{N})_{\tors}(\mathbb{Q})=\begin{cases}
\{O,(-1,0),(0,\pm1),(2,\pm3)\}\cong\mathbb{Z}/6\mathbb{Z}, & \text{if }N=1,\\
\{O,(12,\pm36)\}\cong\mathbb{Z}/3\mathbb{Z}, & \text{if }N=-432,\\
\{O,(0,\pm\sqrt{N})\}\cong\mathbb{Z}/3\mathbb{Z}, & \text{if }N\ne1\text{ is a square},\\
\{O,(-\sqrt[3]{N},0)\}\cong\mathbb{Z}/2\mathbb{Z}, & \text{if }N\ne1\text{ is a cube},\\
\{O\}, & \text{otherwise}.
\end{cases}
\]
\end{lem*}
\begin{proof}
See Exercise 10.19 of \cite{Silverman2009}. 
\end{proof}
\begin{cor}
If $n$ is positive and cube-free, then $C_{-27n^{2}}$ has nontrivial
torsion subgroup (isomorphic to $\mathbb{Z}/3\mathbb{Z}$) if and
only if $n=4$.
\end{cor}
\begin{proof}
Since $n$ is positive and cube-free, $-27n^{2}$ is neither a square
nor a cube.  So $C_{-27n^{2}}$ has no rational points of order $2$,
and the torsion subgroup of $C_{-27n^{2}}$ is nontrivial (and isomorphic
to $\mathbb{Z}/3\mathbb{Z}$) if and only if $-27n^{2}=-432$, i.e.,
$n=4$. 
\end{proof}
Therefore, Proposition \ref{prop:R'(3,1):order} can be strengthened
in the following way. 
\begin{prop}
\label{prop:R'(3,1):rank} The set $\mathscr{R}'(3,1)$ consists of
the special $n=4$ and positive cube-free integers $n$ such that
the rank of $C_{-27n^{2}}$ is positive. 
\end{prop}
\begin{rem*}
If $n=1$, then $C_{-27}$ has rank $0$ and torsion subgroup $\{O,(3,0)\}$.
Thus, $u^{3}+v^{3}=2$ only has one rational solution $(1,1)$. This
gives a modern interpretation of the theorem of Euler mentioned before. 
\end{rem*}
\begin{cor}
If $n=p$, where $p\equiv2\mod9$ is an odd prime number, or $n=p^{2}$,
where $p\equiv5\mod9$ is a prime number, then $n$ is a $(3,1)$-reflecting
number.
\end{cor}
\begin{proof}
This follows directly from Satg\'e's results in \cite{Satge1987}: if
$p\equiv2\mod9$ is an odd prime number, then $C^{2p}$ has infinitely
many rational points, and if $p\equiv5\mod9$ is a prime number, then
$C^{2p^{2}}$ has infinitely many rational points.
\end{proof}

\section{\label{sec:(2,2)}$(2,2)$-Reflecting Numbers}

 We begin with the following property of $(2,2)$-reflecting numbers,
which justify their alternative name, reflecting congruent numbers.
\begin{prop}
A $(2,2)$-reflecting number is a congruent number. 
\end{prop}
\begin{proof}
Let $n$ be a $(2,2)$-reflecting number and $(S,T,U,V)$ be a primitive
solution to the following homogeneous equations 
\begin{equation}
nS^{2}-T^{2}=U^{2},\quad nS^{2}+T^{2}=V^{2}.\label{eq:(2,2)}
\end{equation}
Then, $(-T^{2}/S^{2},TUV/S^{3})$ is a rational point on the congruent
number elliptic curve 
\[
E_{n}:y^{2}=-x(n+x)(n-x)=x(x+n)(x-n)=x^{3}-n^{2}x.
\]
Since $S,T>0$, $V$ is positive. If $U=0$, then $nS^{2}=T^{2}$
and $V^{2}=2T^{2}$, but this is impossible since $\sqrt{2}$ is irrational.
So $TUV/S^{3}\ne0$ and we obtain a non-torsion point on $E_{n}(\mathbb{Q})$.
Hence, $E_{n}$ has a positive rank and $n$ is a congruent number. 
\end{proof}
From now on, each congruent number, if not otherwise specified, is
always assumed to be square-free. 
\begin{notation*}
A triple $(A,B,C)$ of positive integers is called a \emph{Pythagorean
triple} if $A^{2}+B^{2}=C^{2}$; and it is called \emph{primitive}
if in addition $\gcd(A,B,C)=1$. Each rational right triangle is similar
to a unique right triangle with its sides given by a primitive Pythagorean
triple $(A,B,C)$ and is denoted by a triple $(a,b,c)$ of the length
of its sides such that $a/A=b/B=c/C$. 
\end{notation*}
We always assume that $B$ is the even one in any primitive Pythagorean
triple $(A,B,C)$, since $A$ and $B$ must have different parity.
Therefore, we always attempt to make the middle term $b$ in $(a,b,c)$
correspond to the even number $B$ in $(A,B,C)$, unless we are uncertain
about this in some formulas.

Let $\ltriangle_{n}$ be the set of all rational right triangles $(a,b,c)$
with area $n$. Let $\mathscr{P}$ be the set of pairs $(P,Q)$ of
positive coprime integers $P>Q$ with different parity. Then, a primitive
Pythagorean triple can be constructed by Euclid's formula 
\[
(A,B,C)=(P^{2}-Q^{2},2PQ,P^{2}+Q^{2}).
\]
Conversely, any primitive Pythagorean triple is of this form. So $\mathscr{P}$
can be identified with the set of primitive Pythagorean triples. Then,
$PQ(P^{2}-Q^{2})$ and its square-free part $n$ are congruent numbers.
Let $R=\sqrt{PQ(P^{2}-Q^{2})/n}$. Then, $(a,b,c)=(A/R,B/R,C/R)$
 is a rational right triangle in $\ltriangle_{n}$.  Let $\mathscr{P}_{n}\subset\mathscr{P}$
consist of pairs $(P,Q)$ such that the square-free part of $PQ(P^{2}-Q^{2})$
is $n$. 

Let $t\in\mathbb{Q}^{+}$ be such that $n\pm t^{2}$ are both squares.
Since $(-t^{2},\pm t\sqrt{n^{2}-t^{4}})$ are two rational points
on $E_{n}(\mathbb{Q})$, we obtain an injective odd map 
\[
\varphi:\mathscr{T}_{n}:=\{t\in\mathbb{Q}^{*}\mid n\pm t^{2}\in\mathbb{Q}^{*2}\}\to E_{n}(\mathbb{Q});\ t\mapsto(-t^{2},t\sqrt{n^{2}-t^{4}}),
\]
whose image lies in $E_{n}(\mathbb{Q})\setminus(2E_{n}(\mathbb{Q})\cup E_{n}[2])$
(cf. \cite[Prop. 20, §1]{Koblitz2012}), and 
\[
\left(\frac{n^{2}-t^{4}}{t\sqrt{n^{2}-t^{4}}},\frac{2nt^{2}}{t\sqrt{n^{2}-t^{4}}},\frac{n^{2}+t^{4}}{t\sqrt{n^{2}-t^{4}}}\right)\in\ltriangle_{n}.
\]
On the other hand, let $z\in\mathbb{Q}^{+}$ be such that $z^{2}\pm n$
are both rational squares. Since $(z^{2},\pm z\sqrt{z^{4}-n^{2}})$
are two rational points on $E_{n}$, we obtain an injective odd map
\[
\psi:\mathscr{Z}_{n}=\{z\in\mathbb{Q}^{*}\mid z^{2}\pm n\in\mathbb{Q}^{*2}\}\to E_{n}(\mathbb{Q});\ z\mapsto(z^{2},-z\sqrt{z^{4}-n^{2}}),
\]
whose image lies in $2E_{n}(\mathbb{Q})\setminus\{O\}$ (cf. \cite[Prop. 20, §1]{Koblitz2012}).

Denote by $\mathscr{T}_{n}^{+}$ (resp. $\mathscr{Z}_{n}^{+}$) the
subset of positive numbers in $\mathscr{T}_{n}$ (resp. $\mathscr{Z}_{n}$).
Then, there is a one-to-one correspondence between the following sets:
\begin{equation}
\mathscr{Z}_{n}^{+}\leftrightarrow\{(x,y)\in2E_{n}(\mathbb{Q})\setminus\{O\}:y<0\}\leftrightarrow\ltriangle_{n}\leftrightarrow\mathscr{P}_{n},\label{eq:from Z_n+ to Pn}
\end{equation}
where the first correspondence is given by  $\psi$ and its inverse
$\sqrt{x}\mapsfrom(x,y)$, the second one given by (cf. \cite[Prop. 19, §1]{Koblitz2012})
\begin{align*}
(x,y) & \mapsto(\sqrt{x+n}-\sqrt{x-n},\sqrt{x+n}+\sqrt{x-n},2\sqrt{x}),\\
(Z^{2}/4,-|Y^{2}-X^{2}|Z/8) & \mapsfrom(X,Y,Z),
\end{align*}
and the last one given by Euclid's formula and a proper scaling. 

Since  $\left(\frac{n^{2}+t^{4}}{2t\sqrt{n^{2}-t^{4}}}\right)^{2}\pm n=\left(\frac{n^{2}\pm2nt^{2}-t^{4}}{2t\sqrt{n^{2}-t^{4}}}\right)^{2}$,
we obtain an odd map of sets 
\[
z:\mathscr{T}_{n}\to\mathscr{Z}_{n};\ t\mapsto z(t)=\frac{n^{2}+t^{4}}{2t\sqrt{n^{2}-t^{4}}}.
\]
In fact, we can apply the duplication formula 
\[
x([2]P)=\left(\frac{x^{2}+n^{2}}{2y}\right)^{2},\quad\forall P=(x,y)\in E_{n}
\]
to the rational points $P=(-t^{2},\pm t\sqrt{n^{2}-t^{4}})$ on $E_{n}$
and obtain
\[
x([2]P)=\frac{(n^{2}+t^{4})^{2}}{4t^{2}(n^{2}-t^{4})}=z(t)^{2}.
\]
Hence, we have the following commutative diagram 
\begin{equation}
\xymatrix{t\ar@{|->}[d] & \mathscr{T}_{n}\ar[r]^{z}\ar@{^{(}->}[d]_{\varphi} & \mathscr{Z}_{n}\ar@{_{(}->}[d]^{\psi} & z\ar@{|->}[d]\\
(-t^{2},t\sqrt{n^{2}-t^{4}}) & E_{n}(\mathbb{Q})\ar[r]_{[2]} & 2E_{n}(\mathbb{Q}) & (z^{2},-z\sqrt{z^{4}-n^{2}})
}
\label{eq:comm.diag.1}
\end{equation}

\begin{example*}
The first congruent number $n=5$ is also the first reflecting congruent
number. If $t=2$, then $z(t)=41/12$ and $n\pm t^{2},z(t)^{2}\pm n$
are all rational squares.
\end{example*}
Since $z(0+)=z(\sqrt{n}-)=+\infty$, as a real-valued function, certainly
\[
z(t)=\frac{n^{2}+t^{4}}{2t\sqrt{n^{2}-t^{4}}}=\frac{\sqrt{n^{2}-t^{4}}}{2t}+\frac{t^{3}}{\sqrt{n^{2}-t^{4}}}
\]
is not injective for $t\in(-\sqrt{n},0)\cup(0,\sqrt{n})$. However,
as a rational-valued function: 
\begin{prop}
\label{prop:Tn->Zn inj} The map $z:\mathscr{T}_{n}\to\mathscr{Z}_{n}$
is always injective.
\end{prop}
\begin{proof}
If $z\in\mathscr{Z}_{n}$, then there exists $P=(x,y)\in E_{n}(\mathbb{Q})$
such that $x([2]P)=z^{2}$ and the set $\Sigma=\{\pm P,\pm P+T_{1},\pm P+T_{2},\pm P+T_{3}\}$
consists of all rational points $Q\in E_{n}(\mathbb{Q})$ such that
$x([2]Q)=z^{2}$. Then, we have 
\[
x(\pm P+T_{1})=\frac{-n(x-n)}{x+n},\ x(\pm P+T_{2})=\frac{-n^{2}}{x},\ x(\pm P+T_{3})=\frac{n(x+n)}{x-n}.
\]
If $z$ has a preimage $t\in\mathscr{T}_{n}$, then $n-t^{2}=u^{2}$
and $n+t^{2}=v^{2}$ for some $u,v\in\mathbb{Q}^{*}$ and the $x$-coordinates
of points in $\Sigma$ are given by $-t^{2},nv^{2}u^{-2},n^{2}t^{-2},-nu^{2}v^{-2}$.
Since $n$ is not a square, $nu^{2}v^{-2}$ is not a rational square.
So there exists only one $t\in\mathscr{T}_{n}$ as the preimage of
$z$. Hence, $\mathscr{T}_{n}\to\mathscr{Z}_{n}$ is always injective.
\end{proof}
Let $\mathscr{P}_{n}'\subset\mathscr{P}_{n}$ consist of pairs $(P,Q)$
of positive coprime integers $P>Q$ with different parity such that
$P/n$, $Q$, $P-Q$, and $P+Q$ are all integer squares. 
\begin{thm}
A positive square-free integer $n$ is a reflecting congruent number
if and only if the set $\mathscr{P}_{n}'$ is not empty. 
\end{thm}
\begin{proof}
If suffices to show that there is a one-to-one correspondence 
\[
\mathscr{P}_{n}'\leftrightarrow\mathscr{T}_{n}^{+};\quad(P,Q)=(nS^{2},T^{2})\leftrightarrow t=\frac{T}{S}.
\]
For any $(P,Q)\in\mathscr{P}_{n}'$, we write $(P,Q)=(nS^{2},T^{2})$,
where $S,T$ are positive coprime integers, and let $t=T/S=\sqrt{nQ/P}$.
Then, $n\pm t^{2}=n\pm nQ/P=(P\pm Q)/(P/n)$ are all rational squares
and thus $t\in\mathscr{T}_{n}^{+}$. 

Conversely, any $t=T/S\in\mathscr{T}_{n}^{+}$ with $\gcd(S,T)=1$
corresponds to a unique pair $(P,Q)=(nS^{2},T^{2})\in\mathscr{P}_{n}'$.
Indeed, $nS^{2}\pm T^{2}=S^{2}(n\pm t^{2})$ are integer squares.
To show that $P$ and $Q$ are coprime, we consider $d:=\gcd(P,Q)$,
which is given by 
\[
\gcd(nS^{2},T^{2})=\gcd(n,T^{2})=\gcd(n,T),
\]
since $\gcd(S,T)=1$ and $n$ is square-free. If $d$ were greater
than $1$, then $0<U^{2}=nS^{2}-T^{2}$ would imply that $d\mid U$,
$d^{2}\mid nS^{2}$, and thus $d\mid S$, a contradiction.
\end{proof}
\begin{rem*}
 We obtain the following commutative diagram to extend the previous
one 
\begin{equation}
\xymatrix{\mathscr{P}_{n}'\ar@{_{(}->}[r]\ar@{<->}[d] & \mathscr{P}_{n}\ar@{<->}[d]\\
\mathscr{T}_{n}^{+}\ar@{^{(}->}[r]^{z} & \mathscr{Z}_{n}^{+}
}
\label{eq:comm.diag.2}
\end{equation}
So the map $\mathscr{T}_{n}\to\mathscr{Z}_{n}$ is never surjective
unless they are both empty.  For example, if $n$ is reflecting congruent,
for any $z\in\mathscr{Z}_{n}$ such that $\psi(z)\in4E_{n}(\mathbb{Q})$,
$z$ has no preimage in $\mathscr{T}_{n}$, since $[2]\circ\varphi(\mathscr{T}_{n})\subset2E_{n}(\mathbb{Q})\setminus4E_{n}(\mathbb{Q})$. 
\end{rem*}
But a congruent number may not be a reflecting congruent number. 
\begin{lem}
\label{lem:-1QuadRes} If a positive square-free number $n$ is reflecting
congruent, then $-1$ must be a quadratic residue modulo $n$. 
\end{lem}
\begin{proof}
 This is immediate from Proposition \ref{prop:R'(2,1)} since $\mathscr{R}'(2,2)\subset\mathscr{R}'(2,1)$. 
\end{proof}
\begin{lem}
\label{lem:EvenNoReflect} A positive square-free even number $n$
is not reflecting congruent.
\end{lem}
\begin{proof}
Since $n$ is square-free and even, $n\equiv2,6\mod8$. Suppose $(S,T,U,V)$
is a nontrivial primitive solution to the system (\ref{eq:(2,2)}).
If $n\equiv2\mod8$, then the system modulo by $8$ only has the following
solutions 
\[
(S^{2},T^{2},U^{2},V^{2})\equiv(0,0,0,0),(0,4,4,4),(4,0,0,0),\text{ or }(4,4,4,4)\mod8,
\]
which imply that $S,T,U,V$ are all even, a contradiction to $\gcd(S,T,U,V)=1$.
 The case in which $n\equiv6\mod8$ can be proved in the same way
or by Lemma \ref{lem:-1QuadRes} because $n/2\equiv3\mod4$ must have
a prime divisor $\equiv3\mod4$.
\end{proof}
Summarizing the two lemmas, we obtain the following 
\begin{prop}
For a positive square-free number $n$ to be reflecting congruent,
it is necessary that $n$ only has prime divisors $\equiv1\mod4$.
\end{prop}
\begin{example*}
The congruent number $5735=5\cdot31\cdot37$ is not reflecting, as
it has a prime divisor $31\equiv3\mod4$ and thus $-1$ is not a quadratic
residue modulo $5735$\footnote{May MU5735 R.I.P.}. 
\end{example*}
It is natural to consider the problem of reflecting congruent numbers
in local fields $\mathbb{Q}_{v}$, where $v$ is either a prime number
or $\infty$. Recall that for any $a,b\in\mathbb{Q}_{v}^{*}$, the
Hilbert symbol $(a,b)_{v}$ is defined to $1$ if $ax^{2}+by^{2}-z^{2}=0$
has a solution $(x,y,z)$ other than the trivial one $(0,0,0)$ in
$\mathbb{Q}_{v}^{3}$ and $-1$ otherwise. 

Let $n>1$ be a positive square-free integer. If $nS^{2}-T^{2}=V^{2}$
has a nontrivial solution in $\mathbb{Q}_{v}^{3}$, then $(n,-1)_{v}=1$.
If $p$ is an odd prime divisor of $n$, then 
\[
(n,-1)_{p}=\left(\frac{-1}{p}\right)=\begin{cases}
+1, & \text{if }p\equiv1\mod4,\\
-1, & \text{if }p\equiv3\mod4,
\end{cases}
\]
where $\left(\frac{*}{p}\right)$ is the Legendre symbol. Then, $n$
cannot have prime divisors $\equiv3\mod4$, which gives another proof
of Lemma \ref{lem:-1QuadRes}. If $2$ is a divisor of $n$, then
the proof of Lemma \ref{lem:EvenNoReflect} essentially shows that
(\ref{eq:(2,2)}) only has a trivial solution in $\mathbb{Z}_{2}$.
 Hence, $n$ can only has prime divisors congruent to $1$ modulo
$4$. Conversely, 
\begin{lem}
\label{lem:localsol} If a positive square-free number $n$ only has
prime divisors congruent to $1$ modulo $4$, then there exists $t\in\mathbb{Q}_{p}^{*}$
such that $n\pm t^{2}$ are squares in $\mathbb{Q}_{p}$, where $p$
is $\infty$, or $2$, or any prime number congruent to $1$ modulo
$4$, or any prime number congruent to $3$ modulo $4$ such that
$n$ is a quadratic residue modulo $p$. 
\end{lem}
\begin{proof}
Any prime number congruent to $1$ modulo $4$ can be written as a
sum of two distinct positive integer squares. By the Brahmagupta-Fibonacci
identity 
\[
(x^{2}+y^{2})(z^{2}+w^{2})=(xz\mp yw)^{2}+(xw\pm yz)^{2},
\]
$n$ can be written as a sum of two distinct positive integer squares,
say, $n=t^{2}+u^{2}$.

Case 1: $p=\infty$. Clearly, $v^{2}=n+t^{2}$ is a square in $\mathbb{Q}_{\infty}=\mathbb{R}$. 

Case 2: $p=2$. Then, we can choose $t,u$ such that $u^{2}\equiv1\mod8$
and 
\[
t^{2}\equiv\begin{cases}
0\mod8, & \text{if }n\equiv1\mod8,\\
4\mod8, & \text{if }n\equiv5\mod8.
\end{cases}
\]
Then in both cases, $n+t^{2}\equiv1\mod8$ and thus $v^{2}=n+t^{2}$
is a square in $\mathbb{Z}_{2}$. 

Case 3: $p\equiv1\mod4$. If $p\mid n$, then $p\mid t$ if and only
if $p\mid u$; so $p$ cannot divide $t$, otherwise $p^{2}\mid n$.
Therefore, $n+t^{2}\equiv t^{2}\not\equiv0\mod p$ and thus $v^{2}=n+t^{2}$
is a square in $\mathbb{Z}_{p}$. If $p\nmid n$, then we can write
$p=a^{2}+b^{2}$, $p^{2}=x^{2}+y^{2}$ such that $x=a^{2}-b^{2}$
and $y=2ab$, and $np^{2}=t'^{2}+u'^{2}$ such that $p\nmid t'$ and
$p\nmid u'$. Indeed, we have 
\[
np^{2}=(t^{2}+u^{2})(x^{2}+y^{2})=(xt\mp yu)^{2}+(xu\pm yt)^{2}.
\]
If $p$ divides both $xt-yu$ and $xt+yu$, then $p\mid xt$ and $p\mid yu$.
Since $p\nmid x$ and $p\nmid y$, we have $p\mid t$, $p\mid u$,
and thus $p\mid n$, a contradiction. Then, $np^{2}+t'^{2}\equiv t'^{2}\not\equiv0\mod p$
and thus $u^{2}=n-t'^{2}/p^{2}=u'^{2}/p^{2}$ and $v^{2}=n+t'^{2}/p^{2}$
are squares in $\mathbb{Q}_{p}$.

Case 4: $p\equiv3\mod4$ such that $n$ is a quadratic residue modulo
$p$. Let $t=p$. Then, $n\pm t^{2}\equiv n\not\equiv0\mod p$. Thus
$u^{2}=n-t^{2}$ and $v^{2}=n+t^{2}$ are squares in $\mathbb{Z}_{p}$.

In any case, we can find $t$ in $\mathbb{Q}_{p}$ such that $n\pm t^{2}$
are squares in $\mathbb{Q}_{p}$.
\end{proof}
\begin{rem*}
If $p$ is a prime in the residue class $3$ modulo $4$ and $n$
is a quadratic nonresidue modulo $p$, then certain new conditions
that we do not know will be imposed on $n$ so that there exists $t\in\mathbb{Q}_{p}^{*}$
such that $n\pm t^{2}$ are squares in $\mathbb{Q}_{p}$. 
\end{rem*}
Next, we apply a complete $2$-descent on the elliptic curve $E_{n}$
to find a criterion for a congruent number to be reflecting congruent.
The discriminant of $E_{n}$ is $\Delta=64n^{6}$ and the torsion
subgroup of $E_{n}(\mathbb{Q})$ is 
\[
E_{n}[2]=\{T_{0}=O,T_{1}=(-n,0),T_{2}=(0,0),T_{3}=(n,0)\}.
\]
So $E_{n}$ has good reduction except at $2$ and prime divisors of
$n$. Let $S$ be the set of prime divisors of $n$ together with
$2$ and $\infty$. A complete set of representatives for 
\[
\mathbb{Q}(S,2)=\{b\in\mathbb{Q}^{*}/\mathbb{Q}^{*2}:v_{p}(b)\equiv0\pmod2\text{ for all }p\notin S\}
\]
is given by the set $\{\pm\prod_{i}p_{i}^{\varepsilon_{i}}\mid\varepsilon_{i}\in\{0,1\},p_{i}\in S\setminus\{\infty\}\}$.
We identify this set with $\mathbb{Q}(S,2)$. Then, we have the following
injective homomorphism 
\begin{align}
\kappa:E_{n}(\mathbb{Q})/2E_{n}(\mathbb{Q}) & \to\mathbb{Q}(S,2)\times\mathbb{Q}(S,2)\label{eq:kappa}\\
P=(x,y) & \mapsto\begin{cases}
(x-e_{1},x-e_{2}), & \text{if }x\ne e_{1},e_{2},\\
\left(\frac{e_{1}-e_{3}}{e_{1}-e_{2}},e_{1}-e_{2}\right), & \text{if }x=e_{1},\\
\left(e_{2}-e_{1},\frac{e_{2}-e_{3}}{e_{2}-e_{1}}\right), & \text{if }x=e_{2},\\
(1,1), & \text{if }P=O,
\end{cases}\nonumber 
\end{align}
where $e_{1}=-n$, $e_{2}=0$, $e_{3}=n$. A pair $(m_{1},m_{2})\in\mathbb{Q}(S,2)\times\mathbb{Q}(S,2)$,
not in the image of one of the three points $O$, $T_{1}$, $T_{2}$,
is the image of a point $P=(x,y)\in E_{n}(\mathbb{Q})/2E_{n}(\mathbb{Q})$
if and only if the equations 
\begin{equation}
\begin{cases}
e_{2}-e_{1} & =\ n\,=m_{1}y_{1}^{2}-m_{2}y_{2}^{2},\\
e_{3}-e_{1} & =2n=m_{1}y_{1}^{2}-m_{1}m_{2}y_{3}^{2},
\end{cases}\label{eq:2-cover}
\end{equation}
have a solution $(y_{1},y_{2},y_{3})\in\mathbb{Q}^{*}\times\mathbb{Q}^{*}\times\mathbb{Q}$.
If such a solution exists, then $P=(m_{1}y_{1}^{2}+e_{1},m_{1}m_{2}y_{1}y_{2}y_{3})$
is a rational point on $E_{n}$ such that $\kappa(P)=(m_{1},m_{2})$.
 Then, we obtain the first criterion of reflecting congruent numbers: 
\begin{thm}
\label{thm:criterionby2descent} A positive square-free integer $n$
is reflecting congruent if and only if one and thus all of $(1,-1),(2,n),(n,1),(2n,-n)$
lie in the image of $\kappa$. 
\end{thm}
\begin{proof}
Since  $\kappa(T_{1})=(2,-n)$, $\kappa(T_{2})=(n,-1)$, and $\kappa(T_{3})=(2n,n)$,
one of 
\[
(1,-1),(2,n),(n,1),(2n,-n)\in(1,-1)\kappa(E_{n}[2])\subset\mathbb{Q}(S,2)\times\mathbb{Q}(S,2)
\]
lies in the image of $\kappa$ if and only if all of them lie in the
image of $\kappa$. 

If (\ref{eq:(2,2)}) has a nontrivial solution $(S,T,U,V)$ in integers,
then 
\[
(U/S)^{2}+(T/S)^{2}=n,\quad(U/S)^{2}+(V/S)^{2}=2n,
\]
i.e., $(y_{1},y_{2},y_{3})=(U/S,T/S,V/S)$ is a solution to (\ref{eq:2-cover})
for $m_{1}=1$ and $m_{2}=-1$, or in other words, $\kappa(-t^{2},\pm t\sqrt{n^{2}-t^{4}})=(1,-1)$,
where $t=T/S$. 

Conversely, if (\ref{eq:2-cover}) has a solution $(y_{1},y_{2},y_{3})\in\mathbb{Q}^{*}\times\mathbb{Q}^{*}\times\mathbb{Q}$
for $m_{1}=1$ and $m_{2}=-1$, then $P=(x,y)=(y_{1}^{2}-n,-y_{1}y_{2}y_{3})\in E_{n}(\mathbb{Q})$
and 
\[
\sqrt{-x}=\sqrt{n-y_{1}^{2}}=|y_{2}|
\]
is a rational number such that $n-y_{2}^{2}=y_{1}^{2}$ and $n+y_{2}^{2}=y_{3}^{2}$. 
\end{proof}
\begin{cor}
If $n$ is a reflecting congruent number, then $\mathscr{T}_{n}^{+}$
is infinite. 
\end{cor}
\begin{proof}
Indeed, $\mathscr{T}_{n}$ contains at least two elements, say $\pm t_{0}$.
Let $P=\varphi(t_{0})=(-t_{0}^{2},t_{0}\sqrt{n^{2}-t_{0}^{4}})\in E_{n}(\mathbb{Q})$.
Then, $\{\varphi^{-1}(Q)\mid Q\in P+2E_{n}(\mathbb{Q})\}$ is an infinite
subset of $\mathscr{T}_{n}$. Hence, $\mathscr{T}_{n}$ and thus $\mathscr{T}_{n}^{+}$
are infinite. 
\end{proof}
\begin{example*}
Although  $n=205=5\cdot41\equiv5\mod8$ is a square-free congruent
number, which only has prime divisors $\equiv1,5\mod8$, it is not
reflecting. Indeed, $E_{205}(\mathbb{Q})$ has rank $1$ and a torsion-free
generator $(x,y)=(245,2100)$, but one has $\kappa(x,y)=(2,5)$, which
is different from $(1,-1),(2,n),(n,1),(2n,-n)$. 
\end{example*}
Given any $(P,Q)\in\mathscr{P}_{n}$, we get a rational right triangle
$(a,b,c)$ in $\ltriangle_{n}$ of the form $(A/R,B/R,C/R)$, where
$(A,B,C)=(P^{2}-Q^{2},2PQ,P^{2}+Q^{2})$ and $R=\sqrt{PQ(P^{2}-Q^{2})/n}$.
Then, by (\ref{eq:from Z_n+ to Pn}), $(a,b,c)$ corresponds to a
rational point $(c^{2}/4,-|b^{2}-a^{2}|c/8)\in2E_{n}(\mathbb{Q})$,
which by the duplication formula  is the double of the rational point
$(x,y)=(\frac{nb}{c-a},\frac{2n^{2}}{c-a})\in E_{n}(\mathbb{Q})$.\footnote{There are $4$ such rational points in total, but their difference
lies in $E_{n}[2]$ and a different choice of such a point leads to
the same conclusion, that is, $(P,Q)\in\mathscr{P}_{n}'$.} Then, we have  
\begin{align*}
\kappa(x,y) & =\left(\frac{nb}{c-a}+n,\frac{nb}{c-a}\right)=\left(n\frac{B+C-A}{C-A},n\frac{B}{C-A}\right)\\
 & =\left(PQ(P^{2}-Q^{2})\frac{2(P+Q)Q}{2Q^{2}},PQ(P^{2}-Q^{2})\frac{2PQ}{2Q^{2}}\right)\\
 & =(P(P-Q),(P-Q)(P+Q)).
\end{align*}
If $n$ is a reflecting congruent number and $\kappa(x,y)$ is one
of $(1,-1)$, $(2,n)$, $(n,1)$, and $(2n,-n)$, then $\kappa(x,y)$
is either $(2,n)$ or $(n,1)$ since $P>Q>0$.

Case 1: $\kappa(x,y)=(2,n)$. Then, the square-free part of $P^{2}-Q^{2}$
is $n$. It follows that $PQ$ is a square. Since $P$ and $Q$ are
coprime, both of them are squares. Then, the square-free part of $P(P-Q)$
and $P-Q$ could never be $2$, since $P$ and $Q$ have different
parity. So this case never happens.

Case 2: $\kappa(x,y)=(n,1)$. Then, $P^{2}-Q^{2}$ is a square and
the square-free part of $(P-Q)P$ is $n$. It follows that the square-free
parts of $PQ$ is $n$ and $(P+Q)Q$ is a square. Since $P$ and $Q$
are coprime, so are $P+Q$ and $Q$. So $P+Q$ and $Q$ are both squares
and thus $P-Q$ is also a square. Moreover, the square-free part of
$P$ is $n$. 

In a word, we must have $(P,Q)\in\mathscr{P}_{n}'$. Hence, only the
subset of $\mathscr{Z}_{n}^{+}$, which corresponds to $\mathscr{P}_{n}'$
as in \ref{eq:from Z_n+ to Pn}, have preimages in $\mathscr{T}_{n}^{+}$,
which explains (\ref{eq:comm.diag.2}) again.  

Now we can give a description of the image of the map $z:\mathscr{T}_{n}\to\mathscr{Z}_{n}$. 
\begin{prop}
\label{prop:preimage of z} Let $n$ be a reflecting congruent number
and define 
\[
\Sigma(z)=\{P\in E_{n}(\mathbb{Q})\mid[2]P=\pm\psi(z)\}=\{P\in E_{n}(\mathbb{Q})\mid x([2]P)=z^{2}\},
\]
for any $z\in\mathscr{Z}_{n}$. Then, the following conditions are
equivalent: 
\begin{enumerate}
\item \label{enu:preimage} $z\in\mathscr{Z}_{n}$ has a preimage in $\mathscr{T}_{n}$;
\item \label{enu:exist} $\exists P\in\Sigma(z)$, $\kappa(P)\in(1,-1)\kappa(E_{n}[2])$;
\item \label{enu:for all} $\forall P\in\Sigma(z)$, $\kappa(P)\in(1,-1)\kappa(E_{n}[2])$.
\end{enumerate}
So the image of $z:\mathscr{T}_{n}\to\mathscr{Z}_{n}$ is the following
subset of $\mathscr{Z}_{n}$:
\[
\{z\in\mathscr{Z}_{n}\mid\kappa(\Sigma(z))\subset(1,-1)\kappa(E_{n}[2])\}.
\]
\end{prop}
\begin{proof}
(\ref{enu:preimage})$\Rightarrow$(\ref{enu:exist}). If $t\in\mathscr{T}_{n}$
is a preimage of $z\in\mathscr{Z}_{n}$, then $P=(-t^{2},\pm t\sqrt{n^{2}-t^{4}})\in\Sigma(z)$
and $\kappa(P)=(n-t^{2},-t^{2})=(1,-1)$. 

(\ref{enu:exist})$\Leftrightarrow$(\ref{enu:for all}).  For any
$P,Q\in\Sigma(z)$, we have either $P+Q\in E_{n}[2]$ or $P-Q\in E_{n}[2]$,
and thus $\kappa(P)$ and $\kappa(Q)$ lie in the same coset of $\kappa(E_{n}[2])$. 

(\ref{enu:for all})$\Rightarrow$(\ref{enu:preimage}).  Since $\kappa(P)\in(1,-1)\kappa(E_{n}[2])$
for any $P\in\Sigma(z)$, we may replace $P$ by $P+T$ for some $T\in E_{n}[2]$
if necessary so that $\kappa(P)=(1,-1)$. Then, $z$ has a preimage
$\sqrt{-x(P)}$ multiplied by the sign of $z$ in $\mathscr{T}_{n}$. 
\end{proof}

\begin{example*}
Note that $n=41$ is the least prime congruent number such that $E_{n}$
has rank $2$. Let $P=(-9,120)$ be one of the torsion-free generators
of $E_{n}(\mathbb{Q})$. Then, $z=881/120\in\mathscr{Z}_{n}$ is a
square root of $x([2]P)$ and we have 
\[
z^{2}-n=(431/120)^{2},\quad z^{2}+n=(1169/120)^{2}.
\]
However, we have $\kappa(P)=(2,-1)\notin\{(1,-1),(2,41),(41,1),(82,-41)\}$.
So $z$ has no preimage in $\mathscr{T}_{n}$.  Later, we will show
that $n=41$ is also reflecting congruent. 
\end{example*}
Our criterion in Theorem \ref{thm:criterionby2descent} is essentially
based on computing generators for the weak Mordell-Weil group $E_{n}(\mathbb{Q})/2E_{n}(\mathbb{Q})$,
which fits in the short exact sequence
\begin{equation}
0\to E_{n}(\mathbb{Q})/2E_{n}(\mathbb{Q})\to S^{(2)}(E_{n}/\mathbb{Q})\to\sha(E_{n}/\mathbb{Q})[2]\to0,\label{eq:SES}
\end{equation}
where $S^{(2)}(E_{n}/\mathbb{Q})$ is the $2$-Selmer group and $\sha(E_{n}/\mathbb{Q})[2]$
is the $2$-torsion of the Shafarevich-Tate group $\sha(E_{n}/\mathbb{Q})$
of $E_{n}/\mathbb{Q}$. The homogeneous space associated to any pair
$(m_{1},m_{2})\in\mathbb{Q}(S,2)\times\mathbb{Q}(S,2)$ is the curve
in $\mathbb{P}^{3}$ given by the equation 
\[
C_{(m_{1},m_{2})}:\begin{cases}
e_{2}-e_{1} & =\ n\,=m_{1}y_{1}^{2}-m_{2}y_{2}^{2},\\
e_{3}-e_{1} & =2n=m_{1}y_{1}^{2}-m_{1}m_{2}y_{3}^{2},
\end{cases}
\]
and we have the following isomorphism of finite groups: 
\[
S^{(2)}(E_{n}/\mathbb{Q})\cong\{(m_{1},m_{2})\in\mathbb{Q}(S,2)\times\mathbb{Q}(S,2):C_{(m_{1},m_{2})}(\mathbb{Q}_{v})\ne\emptyset,\forall v\in S\},
\]
with which the composition of $E_{n}(\mathbb{Q})/2E_{n}(\mathbb{Q})\to S^{(2)}(E_{n}/\mathbb{Q})$
gives the injective map $\kappa$ in (\ref{eq:kappa}). Since $\kappa(T_{1})=(2,-n)$,
$\kappa(T_{2})=(n,-1)$, and $\kappa(T_{3})=(2n,n)$, we have $E_{n}[2]\cong\{(1,1),(2,-n),(n,-1),(2n,n)\}\subset S^{(2)}(E_{n}/\mathbb{Q})$.
Hence, we have 
\[
\rank E_{n}(\mathbb{Q})=\dim_{\mathbb{F}_{2}}S^{(2)}(E_{n}/\mathbb{Q})-\dim_{\mathbb{F}_{2}}\sha(E_{n}/\mathbb{Q})[2]-2.
\]

If $m_{1}<0$, then $n=m_{1}y_{1}^{2}-m_{2}y_{2}^{2}$ has no solution
in $\mathbb{Q}_{\infty}$ for $m_{2}>0$, and similarly $2n=m_{1}y_{1}^{2}-m_{1}m_{2}y_{3}^{2}$
has no solution in $\mathbb{Q}_{\infty}$ for $m_{2}<0$. Therefore,
we have $(m_{1},m_{2})\notin S^{(2)}(E_{n}/\mathbb{Q})$ whenever
$m_{1}<0$. 

Theorem \ref{thm:criterionby2descent} says that a positive square-free
integer $n$ is reflecting congruent if and only if one and thus all
of $(1,-1),(2,n),(n,1),(2n,-n)$ lie in the image of $E_{n}(\mathbb{Q})/2E_{n}(\mathbb{Q})\to S^{(2)}(E_{n}/\mathbb{Q})$,
or the kernel of $S^{(2)}(E_{n}/\mathbb{Q})\to\sha(E_{n}/\mathbb{Q})[2]$.
This proves the following necessary condition on reflecting congruent
numbers. 
\begin{prop}
\label{prop:necessary condition} For a positive square-free integer
$n$ to be a reflecting congruent number, it is necessary that one
and thus all of $C_{(1,-1)},C_{(2,n)},C_{(n,1)},C_{(2n,-n)}$ have
a point in $\mathbb{Q}_{v}$ for any $v\in S$; It is also sufficient
if $\sha(E_{n}/\mathbb{Q})[2]$ is trivial. 
\end{prop}
Note that this condition covers all the necessary conditions given
before. Indeed, $C_{(1,-1)}:n=y_{1}^{2}+y_{2}^{2},2n=y_{1}^{2}+y_{3}^{2}$
is the same as $n-t^{2}=u^{2},n+t^{2}=v^{2}$. If $-1$ is a quadratic
nonresidue modulo $n$, i.e., $n$ has a prime divisor $p\equiv3\mod4$,
then Lemma \ref{lem:-1QuadRes} essentially says that $C_{(1,-1)}(\mathbb{Q}_{p})$
is empty. If $n$ is even, then Lemma \ref{lem:EvenNoReflect} essentially
says that $C_{(1,-1)}(\mathbb{Q}_{2})$ is empty. Moreover, 
\begin{cor}
\label{cor:p1mod4:(1,-1)} If a positive square-free integer $n$
only has prime divisors congruent to $1$ modulo $4$, then $(1,-1)E_{n}[2]\subset S^{(2)}(E_{n}/\mathbb{Q})$. 
\end{cor}
\begin{proof}
Indeed, Lemma \ref{lem:localsol} says that $C_{(1,-1)}(\mathbb{Q}_{v})$
is not empty for any $v\in S$. 
\end{proof}
\begin{example*}
Return to the example $n=205$. We have $(1,-1)\in S^{(2)}(E_{n}/\mathbb{Q})$
by Corollary \ref{cor:p1mod4:(1,-1)}. Moreover, we have the following
isomorphisms of groups: 
\[
E_{n}(\mathbb{Q})/2E_{n}(\mathbb{Q})\cong(\mathbb{Z}/2\mathbb{Z})^{3},\ S^{(2)}(E_{n}/\mathbb{Q})\cong(\mathbb{Z}/2\mathbb{Z})^{5},\ \sha(E_{n}/\mathbb{Q})[2]\cong(\mathbb{Z}/2\mathbb{Z})^{2},
\]
and in particular $\{(1,\pm1),(1,\pm41),(5,\pm1),(5,\pm41)\}$ gives
a complete set of eight representatives for $S^{(2)}(E_{n}/\mathbb{Q})/E_{n}[2]$.
Since $\kappa(245,2100)\in(1,-41)E_{n}[2]$, the map $\kappa$ has
image $E_{205}[2]\cup(1,-41)E_{205}[2]$, which does not contain $(1,-1)E_{205}[2]$.
\end{example*}
Next we calculate more $2$-Selmer groups for our main results. Here,
we run the routine calculations in detail because we need to know
not only the size of the $2$-Selmer groups but also their group elements.
\begin{lem}
\label{lem:2Selmer(p1mod4)} Let $p$ be a prime integer congruent
to $1$ modulo $4$. Then 
\[
S^{(2)}(E_{p}/\mathbb{Q})=\begin{cases}
(1,\pm1)E_{p}[2]\cup(1,\pm p)E_{p}[2]\cong(\mathbb{Z}/2\mathbb{Z})^{4}, & \text{if }p\equiv1\mod8,\\
(1,\pm1)E_{p}[2]\cong(\mathbb{Z}/2\mathbb{Z})^{3}, & \text{if }p\equiv5\mod8.
\end{cases}
\]
\end{lem}
\begin{proof}
We have $S=\{2,p,\infty\}$ and $\mathbb{Q}(S,2)\cong\{\pm1,\pm2,\pm p,\pm2p\}$.
Since $(m_{1},m_{2})\notin S^{(2)}(E_{p}/\mathbb{Q})$ whenever $m_{1}<0$,
$S^{(2)}(E_{p}/\mathbb{Q})$ is a subgroup of 
\[
S=\{(m_{1},m_{2})\in\mathbb{Q}(S,2)\times\mathbb{Q}(S,2)\mid m_{1}>0\}.
\]
Since $E_{p}[2]\cong\{(1,1),(2,-p),(p,-1),(2p,p)\}$ is a subgroup
of $S^{(2)}(E_{p}/\mathbb{Q})$, 
\[
\{(1,\pm1),(1,\pm2),(1,\pm p),(1,\pm2p)\}
\]
gives a complete set of eight representatives for $S/E_{p}[2]$. 
Since by Corollary \ref{cor:p1mod4:(1,-1)} $(1,-1)\in S^{(2)}(E_{n}/\mathbb{Q})$,
it suffices to check for $m_{2}=2,p,2p$ whether 
\[
C_{(1,m_{2})}:p=y_{1}^{2}-m_{2}y_{2}^{2},\quad2p=y_{1}^{2}-m_{2}y_{3}^{2}
\]
has a point in the local field $\mathbb{Q}_{v}$ for all $v=2,p$,
or equivalently whether 
\[
pZ^{2}+m_{2}Y_{2}^{2}=Y_{1}^{2},\quad2pZ^{2}+m_{2}Y_{3}^{2}=Y_{1}^{2},\quad Z\ne0
\]
has a solution in the ring $\mathbb{Z}_{v}$ of $v$-adic integers
for all $v=2,p$. 

Case 1: $p\equiv1\mod8$. In this case, $-1$ and $2$ are both squares
in $\mathbb{Z}_{p}$. Let $m_{2}=2$ and $Z,Y_{1},Y_{2},Y_{3}\in\mathbb{Z}_{2}$
be such that at least one of them has $2$-adic valuation $0$. Then
$2pZ^{2}+2Y_{3}^{2}=Y_{1}^{2}$ implies that $v_{2}(Y_{1})\ge1$ and
$pZ^{2}+2Y_{2}^{2}=Y_{1}^{2}$ implies that $v_{2}(Z)\ge1$. Then
$2Y_{2}^{2}=Y_{1}^{2}-pZ^{2}$ and $2Y_{3}^{2}=Y_{1}^{2}-2pZ^{2}$
imply that $v_{2}(Y_{2})\ge1$ and $v_{2}(Y_{3})\ge1$, a contradiction.
Let $m_{2}=2p$. Since $p$ is a square in $\mathbb{Z}_{2}$, this
reduces to the previous case. Let $m_{2}=p$. Then, we write $p=a^{2}+b^{2}$,
where $a,b\in\mathbb{Z}$ such that $2$ divides $a$ but not $b$.
Since $-2$ is a square in $\mathbb{Z}_{p}$, we have 
\[
p-2a^{2}\equiv-2a^{2}\not\equiv0\mod p,\quad p-2a^{2}\equiv1-2a^{2}\equiv1\mod8,
\]
which implies that $p-2a^{2}$ is a square in $\mathbb{Z}_{p}$ and
$\mathbb{Z}_{2}$. Let $c$ be any square root of $p-2a^{2}$ in $\mathbb{Z}_{p}$
(resp. $\mathbb{Z}_{2}$). Then, $pZ^{2}+pY_{2}^{2}=Y_{1}^{2}$ and
$2pZ^{2}+pY_{3}^{2}=Y_{1}^{2}$ have a solution $(Z,Y_{1},Y_{2},Y_{3})=(a,p,b,c)$
in $\mathbb{Z}_{p}$ (resp. $\mathbb{Z}_{2}$).  It follows that
$(1,p)\in S^{(2)}(E_{p}/\mathbb{Q})$ and $(1,2),(1,2p)\notin S^{(2)}(E_{p}/\mathbb{Q})$
and hence the assertion. 

Case 2: $p\equiv5\mod8$. In this case, $-1$ is a square but $2$
is not a square in $\mathbb{Q}_{p}$.  Then, the Hilbert symbols
$(p,2)_{p}=(p,2p)_{p}=-1$ imply that $pZ^{2}+2Y_{2}^{2}=Y_{1}^{2}$,
$pZ^{2}+2pY_{2}^{2}=Y_{1}^{2}$, $2pZ^{2}+pY_{3}^{2}=Y_{1}^{2}$ only
have trivial solutions in $\mathbb{Q}_{p}^{3}$. It follows that $(1,2)$,
$(1,2p)$, and $(1,p)$ are not in $S^{(2)}(E_{p}/\mathbb{Q})$, and
hence the assertion. 
\end{proof}
\begin{lem}
\label{lem:2Selmer:Lem5.3:Tian2012} Let $n$ be as in Theorem \ref{thm:TianThm1.3}.
 Then, we have 
\[
S^{(2)}(E_{n}/\mathbb{Q})=(1,\pm1)E_{n}[2]\cong(\mathbb{Z}/2\mathbb{Z})^{3}.
\]
\end{lem}
\begin{proof}
By Corollary \ref{cor:p1mod4:(1,-1)}, we have $(1,-1)\in S^{(2)}(E_{n}/\mathbb{Q})$.
By Lemma 5.3 of \cite{Tian2012}, $S^{(2)}(E_{n}/\mathbb{Q})\cong(\mathbb{Z}/2\mathbb{Z})^{3}$.
Hence, $S^{(2)}(E_{n}/\mathbb{Q})$ must be of the desired form. 
\end{proof}
 Now we present proofs of our main results of this work. 
\begin{proof}[Proof of Theorem \ref{thm:p5mod8}]
 By Proposition \ref{prop:necessary condition} and Corollary \ref{cor:p1mod4:(1,-1)},
it suffices to show that $\sha(E_{p}/\mathbb{Q})[2]$ is trivial.
By Theorem 3.6 of \cite{Monsky1990}, $E_{p}/\mathbb{Q}$ has positive
rank and thus $\dim_{\mathbb{F}_{2}}E_{p}(\mathbb{Q})/2E_{p}(\mathbb{Q})\ge3$.
By Lemma \ref{lem:2Selmer(p1mod4)}, 
\[
\dim_{\mathbb{F}_{2}}S^{(2)}(E_{p}/\mathbb{Q})=3=\dim_{\mathbb{F}_{2}}E_{p}(\mathbb{Q})/2E_{p}(\mathbb{Q})+\dim_{\mathbb{F}_{2}}\sha(E_{p}/\mathbb{Q})[2].
\]
 Hence, $E_{p}(\mathbb{Q})/2E_{p}(\mathbb{Q})\cong S^{(2)}(E_{p}/\mathbb{Q})$
and $\sha(E_{p}/\mathbb{Q})[2]$ is trivial. 
\end{proof}
\begin{example*}
The popular prime congruent number $n=157\equiv5\mod8$ due to Zagier
is reflecting congruent and $E_{157}$ has rank $1$. One can check
that 
\[
t=\frac{407598125202}{53156661805}\mapsto z=\frac{224403517704336969924557513090674863160948472041}{17824664537857719176051070357934327140032961660},
\]
and that $n\pm t^{2}$, $z^{2}\pm n$ are all rational squares. 
\end{example*}
\begin{rem*}
In this example, $z$ seems to be much more complicated than $t$
and thus we have a good reason to study reflecting congruent numbers.
 
\end{rem*}
Return to Conjecture \ref{conj:p1mod8}, by Proposition \ref{prop:necessary condition}
and Corollary \ref{cor:p1mod4:(1,-1)}, it suffices to show that $\sha(E_{p}/\mathbb{Q})[2]$
is trivial. Since $p\equiv1\mod8$ is congruent, by Lemma \ref{lem:2Selmer(p1mod4)},
\[
1\le\rank E_{p}(\mathbb{Q})+\dim_{\mathbb{F}_{2}}\sha(E_{p}/\mathbb{Q})[2]=\dim_{\mathbb{F}_{2}}S^{(2)}(E_{p}/\mathbb{Q})-2=2.
\]
Then, Theorem \ref{thm:p1mod8} is immediate. The parity conjecture
says that $(-1)^{E_{p}(\mathbb{Q})}$ is the global root number of
$E_{p}$ over $\mathbb{Q}$, which is $1$ since $p\equiv1\mod8$.
Thus, we have $\rank E_{p}(\mathbb{Q})\equiv0\mod2$. On the other
hand, if the Shafarevich-Tate group $\sha(E_{p}/\mathbb{Q})$ is finite,
then the order of $\sha(E_{p}/\mathbb{Q})[2]$ is a perfect square;
cf. \cite[Thm. 4.14 of Ch. X]{Silverman2009}. Thus, $\sha(E_{p}/\mathbb{Q})[2]\ne\mathbb{Z}/2\mathbb{Z}$
and $\rank E_{p}(\mathbb{Q})\ne1$. Hence, any one of the parity conjecture
and Shafarevich-Tate conjecture implies that $\rank E_{p}(\mathbb{Q})$
is exactly $2$ and $\sha(E_{p}/\mathbb{Q})[2]$ is trivial. It would
be interesting if one can construct a rational point on $E_{p}$ whose
image under $\kappa$ lies in $(1,-1)E_{p}[2]$ and thus prove Conjecture
\ref{conj:p1mod8} independently of any conjectures.
\begin{example*}
Return to the example $n=41$, the least prime congruent number in
the residue class $1$ modulo $8$. Since $E_{n}(\mathbb{Q})$ has
rank $2$, $n=41$ is reflecting congruent. Indeed, if $t=8/5$, then
$z(t)=1054721/81840\in\mathscr{Z}_{n}$ and 
\begin{align*}
n-t^{2} & =(31/5)^{2}, & z(t)^{2}-n & =(915329/81840)^{2},\\
n+t^{2} & =(33/5)^{2}, & z(t)^{2}+n & =(1177729/81840)^{2}.
\end{align*}
\end{example*}
\begin{rem*}
Therefore, for any prime number $p$, we summarize that 
\begin{enumerate}
\item if $p=2$, then $p$ is not congruent by the infinite descent method; 
\item if $p\equiv1\mod8$ is congruent, then $p$ is conjecturally reflecting
congruent;
\item if $p\equiv3\mod8$, then $p$ is not congruent by the infinite descent
method; 
\item if $p\equiv5\mod8$, then $p$ is reflecting congruent by Theorem
\ref{thm:p5mod8}; 
\item if $p\equiv7\mod8$, then $p$ is not reflecting congruent by Lemma
\ref{lem:-1QuadRes}.
\end{enumerate}
\end{rem*}
\begin{proof}[Proof of Theorems \ref{thm:TianThm1.1}, \ref{thm:TianThm5.2}, and
\ref{thm:TianThm1.3}.]
 We only prove Theorem \ref{thm:TianThm1.3} here. By Proposition
\ref{prop:necessary condition} and Corollary \ref{cor:p1mod4:(1,-1)},
it suffices to show that $\sha(E_{p}/\mathbb{Q})[2]$ is trivial,
which follows from the proof of Theorem 1.3 of \cite{Tian2012}. For
the proof of the other two theorems, see the proof of Theorem 1.1
and Theorem 5.3 of \cite{Tian2012}. 
\end{proof}

\section{\label{sec:gcd(k,m)>=00003D3}$(k,m)$-Reflecting Numbers with $\gcd(k,m)\ge3$}

The goal of this section is to disprove the existence of $(k,m)$-reflecting
numbers with $d=\gcd(k,m)\ge3$. Since the set $\mathscr{R}(k,m)$
of $(k,m)$-reflecting numbers is a subset $\mathscr{R}(d,d)$ of
$(d,d)$-reflecting numbers, it suffices to show that $(d,d)$-reflecting
numbers do not exist for any $d\ge3$. By Theorem \ref{prop:R(k,m)},
$\mathscr{R}(d,d)$ is not empty if and only if the ternary Diophantine
equation 
\begin{equation}
2T^{d}+U^{d}=V^{d}\label{eq:2Td+Ud=00003DVd}
\end{equation}
has an integer solution $(T,U,V)$ such that $-V^{d}<U^{d}<V^{d}$.
Therefore, Theorem \ref{thm:gcd(k,m)>=00003D3} follows immediately
from the following 
\begin{thm}
For any $d\ge3$, the ternary Diophantine equation (\ref{eq:2Td+Ud=00003DVd})
has no integer solution $(T,U,V)$ such that $-V^{d}<U^{d}<V^{d}$. 
\end{thm}
\begin{proof}
If $d=3$, then by the Theorem of Euler in Section \ref{sec:(k,m)},
$2T^{d}=V^{d}+(-U)^{d}$ is not soluble in the set of integers unless
$V=\pm U$. If $d=4$, then another theorem of Euler says that $2T^{d}+U^{d}$
is not a square for any integers $T$ and $U$ unless $T=0$; see
\cite[Ch. XIII, Thm. 210, page 411]{Euler1822elements}. Then, it
suffices to prove the assertion for any odd prime number $d\ge5$.
Then, we rewrite the equation (\ref{eq:2Td+Ud=00003DVd}) as $V^{d}+(-U)^{d}=2T^{d}$,
which, by the D\'enes conjecture, now a theorem, only has integer solutions
$(V,U,T)$ such that $VUT=0$ or $|V|=|U|=|T|$. So equation (\ref{eq:2Td+Ud=00003DVd})
has no integer solution $(T,U,V)$ such that $-V^{d}<U^{d}<V^{d}$. 
\end{proof}
\begin{rem*}
This theorem follows from the Lander, Parkin, and Selfridge conjecture
as well. If $d\ge5$ is odd, then $T,V$ are distinct positive integers.
Clearly, $U$ cannot be $0$. Also, $U$ must be negative; otherwise,
$T,V,U$ will be distinct positive integers and the conjecture implies
$d\le4$. We rewrite $2T^{d}+U^{d}=V^{d}$ as $2T^{d}=V^{d}+(-U)^{d}$.
Since $T,V,-U$ are distinct positive integers, the conjecture implies
that $d\le4$. If $d\ge6$ is even, then we may assume $T,U,V$ to
be distinct and positive. Then, the conjecture also implies that $d\le4$. 
\end{rem*}
\bibliographystyle{alpha}

\newcommand{\etalchar}[1]{$^{#1}$}

\end{document}